\documentclass[12pt,a4paper,reqno]{amsart}

\allowdisplaybreaks
\usepackage{amsmath}
\usepackage{amsfonts}
\usepackage{amssymb,amsthm,amsfonts,amsthm,latexsym,enumerate,url,cases}
\usepackage{graphicx}
\numberwithin{equation}{section}
\usepackage{mathrsfs}
\usepackage{hyperref}
\hypersetup{colorlinks=true,citecolor=blue,linkcolor=blue,urlcolor=blue}
\usepackage[numbers,sort&compress]{natbib}
\usepackage{color}
\def\Z{{\mathbb Z}}

\def\pmod #1{\ ({\rm{mod}}\ #1)}

\def\cP{{\mathcal P}}

\def\Li{{\mathrm{Li}}}
\def\d{{\mathrm{d}}}

\theoremstyle{plain}
\newtheorem{theorem}{Theorem}

\newtheorem{lemma}{Lemma}

\newtheorem{conjecture}{Conjecture}
\theoremstyle{definition}

\newtheorem{remark}{Remark}

\usepackage{etoolbox}
\makeatletter
\patchcmd{\@settitle}{\uppercasenonmath\@title}{}{}{}
\patchcmd{\@setauthors}{\MakeUppercase}{}{}{}
\patchcmd{\section}{\scshape}{}{}{}
\makeatother

\begin{document}
	\title
	[{Primes of the form $ax+by$ in certain intervals with small solutions}]
	{{Primes of the form $ax+by$ in certain  intervals with  small solutions}}
	\author
	[Y. Ding, T.  Komatsu and H. Liu]
	{Yuchen Ding, Takao Komatsu and Honghu Liu}

	\address{(Yuchen Ding) School of Mathematics, Yangzhou University, Yangzhou 225002, People's Republic of China}
	\email{ycding@yzu.edu.cn}
	
	\address{(Takao Komatsu) Institute of Mathematics, Henan Academy of Sciences, Zhengzhou 450046, People's Republic of China?G Department of Mathematics, Institute of Science Tokyo, Meguro-ku, Tokyo 152-8551, Japan} 
	\email{komatsu@zstu.edu.cn;\, komatsu.t.al@m.titech.ac.jp}
	
	\address{(Honghu Liu) School of Mathematics, Yangzhou University, Yangzhou 225002, People's Republic of China}
	\email{3248952439@qq.com }


	\begin{abstract}Let $1<a<b$ be two relatively prime integers and $\Z_{\geq 0}$ the set of non-negative integers. For any non-negative integer $\ell$, denote by $g_{\ell,a,b}$ the largest integer $n$ such that the equation 
	\begin{equation}\label{abs-1}
		n=ax+by,\quad  (x,y)\in\Z_{\geq 0}^{2} \tag{1}
	\end{equation}
	has at most $\ell$ solutions. Let $\pi_{\ell,a,b}$ be the number of primes $p\leq g_{\ell,a,b}$ having at least $\ell+1$ solutions for \eqref{abs-1} and $\pi(x)$ the number of primes not exceeding $x$. In this article, we prove that 
	for a fixed integer $a\ge 3$ with $\gcd(a,b)=1$, 
	$$
	\pi_{\ell,a,b}=\left(\frac{a-2}{2(\ell a+a-1)}+o(1)\right)\pi\bigl(g_{\ell,a,b}\bigr)\quad(\text{as}~ b\to\infty).
	$$ 
	For any non-negative $\ell$ and relatively prime integers $a,b$, satisfying $e^{\ell+1}\leq a<b$, we show that
	\begin{equation*}
		\pi_{\ell,a,b}>0.005\cdot \frac{1}{\ell+1}\frac{g_{\ell,a,b}}{\log g_{\ell,a,b}}.
	\end{equation*} 
	Let $\pi_{\ell,a,b}^{*}$ be the number of primes $p\leq g_{\ell,a,b}$ having at most $\ell$ solutions for \eqref{abs-1}.
	For an integer $a\ge 3$ and a large sufficiently integer $b$ with $\gcd(a,b)=1$, we also prove that
	$$   
	\pi^{*}_{\ell,a,b}>\frac{(2\ell+1)a}{2(\ell a+a-1)}\frac{g_{\ell,a,b}}{\log g_{\ell,a,b}}.
	$$ 
	Moreover, if  $\ell<a<b$ with $\gcd(a,b)=1$, then we have
		\begin{equation*}
		\pi^{*}_{\ell,a,b}>\frac{\ell+0.02}{\ell+1}\frac{g_{\ell,a,b}}{\log g_{\ell,a,b}}.
		\end{equation*}
These results generalize the previous ones of Chen and Zhu (2025), who established the results for the case $\ell=0$.

	\end{abstract}
	\maketitle

	\section{Introduction}
Let $a<b$ be two relatively prime positive integers and $\Z_{\geq 0}$ the set of non-negative integers. We denote by $\gcd(a,b)$ the greatest common divisor of $a$ and $b$.
For a non-negative integer $\ell$, the $\ell$-numerical semigroup, denoted by $S_\ell(a,b)$, is the set of the non-negative integers $n$ whose number of integral representations $ax+by=n\, (x,y\in\Z_{\geq 0})$ is more than $\ell$. The set $\ell$-gaps  $G_{\ell}(a,b):=\mathbb{Z}_{\geq 0 }\backslash S_{\ell}(a,b)$ is the set of the non-negative integers $n$ whose number of integral representations $ax+by=n\, (x,y\in\Z_{\geq 0})$ is less than or equal to $\ell$. 
The $\ell$-Frobenius number $g_{\ell,a,b}$ is the largest element in $G_{\ell}(a,b)$ and the $\ell$-genus $n_{\ell,a,b}$ (or $\ell$-Sylvester number) is the number of the elements in $G_{\ell}(a,b)$. Sylvester \cite{Sylvester} proved that $g_{0,a,b}=ab-a-b$, and further showed that exactly one of $n$ and $g_{0,a,b}-n$ belongs to $S_{0}(a,b)$ for any non-negative integer $n\leq g_{0,a,b}$. It is obvious that exactly half of the integers in the interval $[0,g_{0,a,b}]$ can be expressed in the form $ax+by$ with $(x,y)\in\Z_{\geq 0}^{2}$, according to Sylvester's result.
More generally, for any pair of relatively prime integers $(a,b)$, we have (see, e.g., \cite[Corollary 2]{Kom2}) 
\begin{align*}
	g_{\ell,a,b}&=(\ell+1)a b-a-b\,,
	\label{p-frob}\\
	n_{\ell,a,b}&=\frac{1}{2}\bigl((2\ell+1)a b-a-b+1\bigr)\,.
\end{align*}

When $\ell=0$, the $0$-numerical semigroup $S(a,b)=S_0(a,b)$ is the original numerical semigroup. In this case, the $0$-Frobenius number $g_{a,b}=g_{0,a,b}$ and the $0$-genus $n_{a,b}=n_{0,a,b}$ are the original Frobenius number and genus, respectively.  
There are exactly $n_{0,a,b}=(a-1)(b-1)/2$ integers in the set $G_0(a,b)$, as previously mentioned. Let $\mathcal P$ be the set of the prime numbers.  
For any given non-negative integer $\ell$, denote 
$$
\pi_{\ell,a,b}=\#\{p\in\mathcal P: p\in S_{\ell}(a,b), p\leq g_{\ell,a,b}\}\,.  
$$
In 2020,
Ram\'{\i}rez Alfons\'{\i}n and Ska{\l}ba \cite{RS} proved that for any $\varepsilon>0$, there exists a constant $C(\varepsilon)>0$ such that 
\begin{equation}\label{intr-1}
	\pi_{0,a,b}>C(\varepsilon)\frac{g_{0,a,b}}{(\log g_{0,a,b})^{2+\varepsilon}}
\end{equation}
provided that $g_{0,a,b}$ is sufficiently large.
Observing (\ref{intr-1}) and
verifying several cases, they posed the following conjecture \cite[Conjecture 2]{RS}.
\begin{conjecture}\label{con-1}
	Let $2<a<b$ be two relatively prime integers. Then $\pi_{0,a,b}>0$.
\end{conjecture}
Dai, Ding, and Wang \cite{Dai/Ding/Wang} confirmed Conjecture \ref{con-1}. Recently, Chen and Zhu \cite{Chen/Zhu} proved that
\begin{equation*}
	\pi_{0,a,b}>0.005 \cdot \frac{g_{0,a,b}}{\log g_{0,a,b}}.
\end{equation*} 

Ram\'{\i}rez Alfons\'{\i}n and Ska{\l}ba also posed another conjecture \cite[Conjecture 3]{RS}.
\begin{conjecture}\label{con-2}
	Let $2<a<b$ be two relatively prime integers; then
	$$
	\pi_{0,a,b}\sim \frac{\pi(g_{0,a,b})}{2} \quad (\text{as}\, a\to \infty),
	$$
	where $\pi(x)$ denotes the number of primes not exceeding $x$.
\end{conjecture}

 Ding \cite{Ding} proved that Conjecture \ref{con-2} is almost always true. Subsequently, 
 Ding, Zhai, and Zhao \cite{Ding/Zhai/Zhao} confirmed Conjecture \ref{con-2}. More generally, Ding and Komatsu \cite{Ding/Komatsu} proved that 
\begin{equation}\label{intr-Ding/Komatsu-1}
	\pi_{\ell,a,b}\sim \frac{1}{2} \frac{\pi(g_{\ell,a,b})}{\ell+1} \quad (\text{as}\, a\to \infty )
\end{equation}
as a generalization of Conjecture \ref{con-2}. Huang and Zhu \cite{Huang/Zhu} further extended the result of Ding, Zhai, and Zhao to the distribution of prime powers. Very recently, Chen and Zhu \cite{Chen/Zhu} also proved that for any fixed integer $a\geq 3$ with $\gcd(a,b)=1$, then
\begin{equation*}
	\pi_{0,a,b}=\left(\frac{a-2}{2(a-1)}+o(1) \right)\frac{g_{0,a,b}}{\log g_{0,a,b}} \quad (as \,b\to \infty).
\end{equation*}

In this paper, we give a generalization of the results of Chen and Zhu  \cite{Chen/Zhu}.

\begin{theorem}\label{th1}
	Let $\ell$ be any non-negative integer. Then, for a fixed integer $a\ge 3$ with $\gcd(a,b)=1$, we have
	$$
	\pi_{\ell,a,b}=\left(\frac{a-2}{2(\ell a+a-1)}+o(1)\right)\pi\bigl(g_{\ell,a,b}\bigr)\quad(\text{as}~ b\to\infty)\,. 
	$$ 
\end{theorem}

\begin{theorem}\label{thm-1}
	Let $\ell$ be a non-negative integer and $a<b$ two relatively prime integers satisfying $a\geq e^{\ell+1}$, then
	\begin{equation*}
		\pi_{\ell,a,b}> 0.005 \cdot \frac{1}{\ell+1}\frac{g_{\ell,a,b}}{\log g_{\ell,a,b}}.
	\end{equation*}
	\begin{remark}
		For $\ell=0$, Theorem \ref{th1} and Theorem \ref{thm-1} were proved by Chen and Zhu \cite{Chen/Zhu} as mentioned earlier. By (\ref{intr-Ding/Komatsu-1}) and the prime number theorem, we see that $\ell+1$ in the denominator is required.
	\end{remark}
\end{theorem}

Chen and Zhu were also interested in determining the number of primes in $G_{0}(a,b)$.
For any given non-negative integer $\ell$, denote 
$$
\pi^{*}_{\ell,a,b}=\#\{p\in\mathcal{P}: p\in G_{\ell}(a,b)\}.  
$$
Chen and Zhu \cite{Chen/Zhu-2} proved that for a fixed integer $a\ge 3$ with $\gcd(a,b)=1$, then
\begin{equation*}
	\pi_{0,a,b}^{*} >\frac{a}{2(a-1)}\frac{g_{0,a,b}}{\log g_{0,a,b}} \quad (as\, b\to\infty),
\end{equation*}
and
\begin{equation*}
	\pi_{0,a,b}^{*}> 0.04\pi(g_{0,a,b}).
\end{equation*}
Chen and Zhu also posed the following conjecture \cite[Conjecture 1.4]{Chen/Zhu-2}. 
\begin{conjecture}
	For $b>a\geq 1$ with $\gcd(a,b)=1$, we have
	$$
	\pi_{0,a,b}^{*}\geq \frac{1}{2}\pi(g_{0,a,b}), 
	$$
	and the equality holds if and only if one of the following four cases holds:
	\begin{equation*}
		(1)\, a=1; \quad (2)\, (a,b)=(2,3); \quad (3)\, (a,b)=(2,5); \quad (4)\,(a,b)=(3,5). 
	\end{equation*}
\end{conjecture}

Along this line, we also have the following results.
\begin{theorem}\label{th:22}
	Let $\ell$ be a non-negative integer.
	Then, for a given integer $a\ge 3$ and a sufficiently large integer $b$ with $\gcd(a,b)=1$, we have
	$$   
	\pi^{*}_{\ell,a,b}>\frac{(2\ell+1)a}{2(\ell a+a-1)}\frac{g_{\ell,a,b}}{\log g_{\ell,a,b}}.
	$$ 
\end{theorem}  
\begin{theorem}\label{thm-4}
	Let $\ell,a$ and $b$ be non-negative integers with $\ell<a<b$ and $\gcd(a,b)=1$. Then
	\begin{equation*}
		\pi^{*}_{\ell,a,b}>  \frac{\ell+0.02}{\ell+1}\frac{g_{\ell,a,b}}{\log g_{\ell,a,b}}.
	\end{equation*}
\end{theorem}
\begin{remark}
	For $\ell=0$, Theorem \ref{th:22} and Theorem \ref{thm-4} were proved by Chen and Zhu \cite{Chen/Zhu-2} as mentioned previously. By (\ref{intr-Ding/Komatsu-1}) and the prime number theorem, it follows that
	$$
	\pi^{*}_{\ell,a,b}\sim \frac{\ell+0.5}{\ell+1} \frac{g_{\ell,a,b}}{\log g_{\ell,a,b}} \quad (as \, a\to\infty),
	$$ and hence the constant $\frac{\ell}{\ell+1}$ is necessary in the statement of Theorem  \ref{thm-4}.
\end{remark}

\section{Proofs} 
 From now on, we assume that $\ell\geq 1$ and use $g$ and $S$ instead of $g_{\ell,a,b}$ and $S_{\ell}(a,b)$ for simplicity.
The first lemma is quoted from Ding and Komatsu \cite[Lemma 2]{Ding/Komatsu}.
 \begin{lemma}\label{lem-5-1}
 	Suppose that the number of solutions to
 	$$ n=ax+by, \quad (x,y)\in \Z_{\geq 0}^{2}$$
 	is exactly $\ell+1$. Then we have 
 	$$
 	n=\ell ab+ax_{0}+by_{0}
 	$$
 	for some $0\leq x_{0}\leq b-1,0\leq y_{0}\leq a-1$.
 \end{lemma}
  Combining Lemma \ref{lem-5-1} with Sylvester's result as mentioned previously, we obtain the following claim.
 \begin{lemma}\label{lem-5}
 Let $a$ and $b$ be relatively prime positive integers. Then, any non-negative integer $n\leq g_{\ell,a,b}$ has exactly $\ell+1$ solutions of $n=ax+by$ with $(x,y)\in\Z_{\geq 0}^{2}$, if and only if, $g_{\ell,a,b}-n$ cannot be expressed in the form $ax+by$ with $x,y\in\Z_{\geq 0}$. 
 \end{lemma}
 
Let $\pi(x;k,l)$ denote the number of primes $p\leq x$ with $p\equiv l \pmod{k}$ and $\varphi(m)$ the Euler totient function of $m$. The third lemma is known as the Siegel--Walfisz Theorem (see \cite[Chapter 22]{Davenport1980}).
 
 \begin{lemma} (Siegel--Walfisz Theorem) \label{pi(x;m,l)=} Let $A$ be a positive real number and $m, l$ two coprime
 	integers with $1\le m\le (\log x)^A$. Then there is a positive
 	constant $D$ depending only on $A$ such that
 	$$\pi (x;m,l)=\frac 1{\varphi (m)} \int_2^x \frac 1{\log t} \d t +O\left(x \exp \big(-D \sqrt{\log x}\big)\right), $$
 	uniformly in $m$.
 \end{lemma}
 
  The following lemma is quoted from Bennett, Martin, O'Bryant, and Rechnitzer \cite[(1.15)]{Bennett/Martin} .
 \begin{lemma}\label{lem-1} For $x\ge 1474279333$, we have
 	$$\big|\pi (x)-\Li (x)\big|<0.0008375\frac{x}{\log^2 x}, $$
 	where
 	$$\Li (x)=\int_2^x \frac{\d t}{\log t} .$$
 \end{lemma}
 As an application of Lemma \ref{lem-1}, we have the following result. 
 
 \begin{lemma}\label{lem-2}
 	Let $\ell$ be a positive integer and $x\geq \max\{\ell e^{2(\ell+1)},7\cdot 10^{5}\}$. Then
 	\begin{equation*}
 		\pi(x)-\pi(cx)\geq 0.11604925 \cdot \frac{1}{\ell+1} \frac{x}{\log x},
 	\end{equation*}
 	where $c=1-\frac{1}{8(\ell+1)}$.
 \end{lemma}
 
 \begin{proof}
 The proofs will be separated into two cases.
 	
 	\textbf{Case 1.}
 	For $x\geq \max\{\ell e^{2(\ell+1)},1.573\cdot 10^{9}\}$, we have 
 	$$
 	cx\geq \left( 1-\frac{1}{16} \right)x\geq \left( 1- \frac{1}{16}\right)\cdot 1.573\cdot 10^{9} > 1474279333.
 	$$
 	By Lemma \ref{lem-1}, we have
 	\begin{align}\label{lem2-ineq-1}
 		\pi(x)-\pi(cx)
 		&> \Li(x)-0.0008375\frac{x}{(\log x)^{2}}-\Li(cx)-0.0008375\frac{cx}{\big(\log (cx)\big)^{2}}\nonumber\\
 		&=\int_{cx}^{x}\frac{\d t}{\log t}-0.0008375\frac{x}{(\log x)^{2}}-0.0008375\frac{cx}{(\log c+\log x )^{2}}\nonumber\\
 		&\geq (1-c)\frac{x}{\log x}-0.0008375\frac{x}{(\log x)^{2}}-0.0008375\frac{cx}{(\log c+\log x)^{2}}\nonumber\\
 		&= \left( (1-c)-0.0008375\frac{1}{\log x}-0.0008375 \frac{c\log x}{(\log c+\log x)^{2}}\right)\frac{x}{\log x}\nonumber\\
 		&= \left( \frac{1}{8(\ell+1)}-\frac{0.0008375}{\log x}-\frac{0.0008375 c}{(\log c)^{2}/\log x+2\log c+\log x}\right)\frac{x}{\log x}.
 	\end{align}
 	Recalling that $c=1-\frac{1}{8(\ell+1)}$ and $x\geq \max\{\ell e^{2(\ell+1)},1.573\cdot 10^{9}\}$, we have 
 	\begin{align*}
 		\frac{(\log c)^{2}}{ \log x}+2\log c+ \log x
 		>2\log\left(1-\frac{1}{16}\right)+\log 10^{9}
 		>20
 		> 2 (\ell+1), \quad (1\leq \ell \leq 8)
 	\end{align*}
 	and
 	\begin{align*}
 		\frac{(\log c)^{2}}{\log x}+2\log c+ \log x 
 		&> 2\log \left(1-\frac{1}{8\cdot (9+1)}\right)+\log\left( \ell e^{2(\ell+1)} \right)\\
 		&>-0.025157565 +\log \ell+2(\ell+1)\\
 		&>2(\ell+1), \quad (\ell\geq 9).
 	\end{align*}
 	So, for any positive integer $\ell$ we have
 	\begin{equation}\label{lem2-ineq-2}
 		\frac{c}{(\log c)^{2}/\log x+2\log c+\log x}< \frac{1}{2} \frac{c}{ \ell+1}<\frac{1}{2}\frac{1}{\ell+1}.
 	\end{equation}
 	Hence, by (\ref{lem2-ineq-1}), $x\geq \ell e^{2(\ell+1)}$, and (\ref{lem2-ineq-2}), we have
 	\begin{align*}
 		\pi(x)-\pi(cx)
 		&> \left(\frac{1}{8(\ell+1)}-\frac{0.0008375}{2(\ell+1)}-\frac{0.0008375}{2(\ell+1)}\right)\frac{x}{\log x}\\
 		&\geq 0.1241625 \cdot \frac{1}{\ell+1}\frac{x}{\log x}.
 	\end{align*}
 	
 	\textbf{Case 2.} For $\max\{\ell e^{2(\ell+1)},7\cdot 10^{5}\}\leq x< 1.573\cdot 10^{9}$, there exist two integers $u,k$ with
 	$2\le u\le 6$ and  $10^3\le k<10^4$ such that $10^uk\le
 	x<10^u(k+1)$. With the help of a computer\footnote{In Case 2, we employed a Python program with the Python packages math and sympy to iterate over all integers $u$ and $k$ satisfying $10^{u}k\leq x<10^{u}(k+1)$, with the aim of finding the minimum value of the coefficient of $\frac{1}{\ell+1}$.}, we have
 	\begin{align*}
 		\big(\pi(x)-\pi(cx)\big)\frac{\log x}{x}
 		\geq &\Big( \pi(10^{u}k)-\pi \big(c \cdot 10^{u}(k+1) \big) \Big)\frac{\log \left( 10^{u}(k+1) \right)}{10^{u}(k+1)}\\
 		\geq &0.11604925\ldots \cdot \frac{1}{\ell+1},
 	\end{align*}
 	and the last equality holds if $\ell=8$, $u=6$, and $k=1001$.
 	
 	This completes the proof of Lemma \ref{lem-2}.
 \end{proof}
 
  We need the lemma quoted from Montgomery and Vaughan \cite[Theorem 2]{Montgomery/Vaughan}. 
 \begin{lemma} \label{lem-3}
 	Let $k,l$ be two integers and let $x,y$ be two positive real numbers
 	with $1\le k<y$. Then
 	$$\pi (x+y; k,l)-\pi (x; k,l)<\frac{2y}{\varphi (k) \log (y/k)}.$$
 \end{lemma}
 
   We need the following lemma quoted from Chen and Zhu \cite[Lemma 2.8]{Chen/Zhu}.
 \begin{lemma}\label{lem-4}
 	Let $m,k$ be two relatively prime integers. Then for any positive real number $x$ and any integers $l$, we have
 	\begin{equation*}
 		\sum_{\stackrel{1\leq n\leq x}{\gcd(nk+l,m)=1}}1=\frac{\varphi(m)}{m}x+\theta_{x,m,k,l} 2^{\omega(m)},
 	\end{equation*}  
 	where $\theta_{x,m,k,l}$ is a real number with $\lvert \theta_{x,m,k,l}\rvert \leq 1$ and $\omega(m)$ the number of distinct prime divisors of $m$.
 \end{lemma}
 
 The following lemma is an application of Lemma \ref{lem-4}.
 \begin{lemma}\label{lem-6}
 	Let $b>a\geq 10$ be two relatively prime integers. Then we have
 	\begin{align*}
 		\#&\left\{p\in\cP:cg< p\leq g, p\notin S \right\}\\
 		&\quad\quad\quad\quad\quad	< \frac{1}{4(\ell+1)\varphi(a)} \Bigg(\sum_{\stackrel{1\leq y\leq a/8+1}{\gcd(y,a)=1} }1\Bigg) \left( 1-\frac{\log \big(8(\ell+1)a \big)}{\log g}\right)^{-1}\cdot \frac{g}{\log g}\\
 		&\quad\quad\quad\quad\quad	<  \frac{1}{4(\ell+1)} \left( \frac{1}{8}+\frac{1}{a}+\frac{2^{\omega(a)}}{\varphi(a)} \right) \left( 1-\frac{\log \big(8(\ell+1)a \big)}{\log g}\right)^{-1}\cdot \frac{g}{\log g},
 	\end{align*}
 	where $\ell\geq 1$ and $c=1-\frac{1}{8(\ell+1)}$.
 \end{lemma}
 \begin{remark}
 	If $a>6\cdot 10^{4}$, we have (see \cite[p.10]{Chen/Zhu})
 	\begin{equation}\label{C1-ineq-1}
 		\frac{2^{\omega(a)}}{\varphi(a)}<0.00556112.
 	\end{equation}
 \end{remark}
 \begin{proof}
 	We appoint that $x$ and $y$ are always non-negative integers. By Lemma \ref{lem-5}, we have
 	\begin{align}\label{lem-6-1}
 		\#&\left\{p\in\cP:cg< p\leq g, p\notin S \right\}\nonumber\\
 		&\quad\quad\quad\quad\quad\quad=\#\left\{p\in\cP:cg< p\leq g, g-p=ax+by,0\leq by\leq (1-c)g\right\}\nonumber\\
 		&\quad\quad\quad\quad\quad\quad\leq \#\left\{p\in \cP: cg< p\leq g,p\equiv -b-by \pmod{a}, 0\leq by\leq (1-c)g \right\}\nonumber\\
 		&\quad\quad\quad\quad\quad\quad=\sum_{\stackrel{0\leq y\leq (1-c)g/b}{\gcd(1+y,a)=1}} \Big( \pi(g;a,-b-by)-\pi(cg;a,-b-by) \Big)\nonumber\\
 		&\quad\quad\quad\quad\quad\quad=\sum_{\stackrel{1\leq y\leq (1-c)g/b+1}{\gcd(y,a)=1} }\Big( \pi(g;a,-by)-\pi(cg;a,-by) \Big).
 	\end{align}
 	Since $g=(\ell+1)ab-a-b$, $b>a\geq 10$, and $c=1-\frac{1}{8(\ell+1)}$, we have 
 	$$
 	g-cg=\frac{1}{8(\ell+1)}g=\frac{ab}{8}-\frac{a+b}{8(\ell+1)}
 	>\left( \frac{b}{8}-\frac{1}{16}-\frac{b}{16a}\right)a
 	>a.
 	$$
 	It follows that
 	\begin{equation}\label{lem-6-2}
 		\pi(g;a,-by)-\pi(cg;a,-by)<\frac{2(1-c)g}{\varphi(a)\log\big((1-c)g/a\big)}
 	\end{equation}
 	from $g-cg>a$ and Lemma \ref{lem-3}.
 	Hence, by \eqref{lem-6-1}, \eqref{lem-6-2}, and $1-c=\frac{1}{8(\ell+1)}$ we have
 	\begin{align}\label{lem-6-3}
 	 \#&\left\{p\in\cP:cg< p\leq g, p\notin S \right\}\nonumber\\
 		&\quad\quad\quad<\sum_{\stackrel{1\leq y\leq (1-c)g/b+1}{\gcd(y,a)=1} } \frac{2(1-c)g}{\varphi(a)\log\big((1-c)g/a\big)}\nonumber\\
 		&\quad\quad\quad=\frac{1}{4(\ell+1)\varphi(a)}\sum_{\stackrel{1\leq y\leq (1-c)g/b+1}{\gcd(y,a)=1} } \frac{g}{\log g-\log \big(8(\ell+1)a\big)}\nonumber\\
 		&\quad\quad\quad=\frac{1}{4(\ell+1)\varphi(a)}\Bigg(\sum_{\stackrel{1\leq y\leq (1-c)g/b+1}{\gcd(y,a)=1}}1\Bigg)\left(1-\frac{\log\big(8(\ell+1)a\big)}{\log g}\right)^{-1}\cdot\frac{g}{\log g}.
 	\end{align}
 	Recalling that $g=(\ell+1)ab-a-b$ and $1-c=\frac{1}{8(\ell+1)}$, we have
 	\begin{align}\label{lem-6-4}
 	\frac{(1-c)g}{b}=\frac{(\ell+1)ab-a-b}{8(\ell+1)b}<\frac{a}{8}.
 	\end{align}
 	Therefore, by \eqref{lem-6-3}, \eqref{lem-6-4}, and Lemma \ref{lem-4}, we have
 	\begin{align*}
 		\#&\left\{p\in\cP:cg< p\leq g, p\notin S \right\}\\
 	&\quad\quad\quad\quad\quad	< \frac{1}{4(\ell+1)\varphi(a)} \Bigg(\sum_{\stackrel{1\leq y\leq a/8+1}{\gcd(y,a)=1} }1\Bigg) \left( 1-\frac{\log \big(8(\ell+1)a \big)}{\log g}\right)^{-1}\cdot \frac{g}{\log g}\\
 	&\quad\quad\quad\quad\quad	<  \frac{1}{4(\ell+1)} \left( \frac{1}{8}+\frac{1}{a}+\frac{2^{\omega(a)}}{\varphi(a)} \right) \left( 1-\frac{\log \big(8(\ell+1)a \big)}{\log g}\right)^{-1}\cdot \frac{g}{\log g},
 	\end{align*}
 	which completes the proof of Lemma \ref{lem-6}.
 \end{proof}
 
 We also need the following lemma quoted by Bennett, Martin, O'Bryant, and Rechnitzer \cite[Corollay 1.6]{Bennett/Martin}.
 \begin{lemma}\label{lem-7}
 	Let $m,l$ be two relatively prime integers with $1\leq m \leq 1200$. Then for all $x\geq 50 m^{2}$, we have
 	\begin{equation*}
 		\frac{x}{\varphi(m)\log x}<\pi(x;m,l)<\frac{x}{\varphi(m)\log x}\left( 1+\frac{5}{2\log x}\right).
 	\end{equation*}
 \end{lemma}
 
 The following lemma quoted from Chen and Zhu \cite[Corollary 2.3]{Chen/Zhu-2} is a variant of Lemma \ref{lem-7}.
 
 \begin{lemma}\label{pi(x;m,l)}
 	Let $m,l$ be two relatively prime integers. Then there is $X(m)$ such that for all $x\geq X(m)$,
 	\begin{equation*}
 		\frac{x}{\varphi(m)\log x}<\pi(x;m,l)<\frac{x}{\varphi(m)\log x}\left(1+\frac{5}{2\log x}  \right).
 	\end{equation*}
 \end{lemma}
  The final lemma is quoted from Rosser and Schoenfeld \cite[Corollary 1]{RosserSchoenfeld1962}.
 \begin{lemma}\label{pi(x)estimate} 
 	 For $x\ge 17$ we have
 	$\pi (x)> x/\log x$.
 \end{lemma}
\subsection{Proof of Theorem \ref{th1}}
 \begin{proof}[Proof of Theorem \ref{th1}]
 	By Lemmas \ref{lem-5-1} and \ref{pi(x;m,l)=}, for $\ell\geq 1$ we have
 	\begin{align*} 
 		\pi_{\ell,a,b}
 		&=\sum_{\substack{\ell a b<p<g\\p=\ell a b+a x+b y\\x,y\in\mathbb Z_{\ge 0}}}1
 		=\sum_{\substack{y=1\\ \gcd(y,a)=1}}^{a-1}\sum_{\substack{p\equiv b y\pmod{a}\\\ell a b+b y\le p<g}}1\\
 		&=\sum_{\substack{y=1\\\gcd(y,a)=1}}^{a-1}\big(\pi(g;a,b y)-\pi(\ell a b+b y; a, b y)\big)+O_{\ell,a}(1)\\ 
 		&=\sum_{\substack{y=1\\\gcd(y,a)=1}}^{a-1}\left(\frac{1}{\varphi(a)}\int_2^{g}\frac{\d t}{\log t}
 		-\frac{1}{\varphi(a)}\int_2^{\ell a b+b y}\frac{\d t}{\log t}+O\big(g e^{-D\sqrt{\log g}}\big) \right)\\ 
 		&=\int_2^{g}\frac{\d t}{\log t}-\frac{1}{\varphi(a)}\sum_{\substack{y=1\\\gcd(y,a)=1}}^{a-1}\int_2^{\ell a b+b y}\frac{\d t}{\log t}+O\big(g e^{-D\sqrt{\log g}}\big)\,.
 	\end{align*}  
 	Given $1\le y\le a-1$. Integrating by parts, we obtain
 	\begin{align}\label{thm-1-eq-1}
 	\int_2^{\ell a b+b y}\frac{\d t}{\log t}=\frac{\ell a b+b y}{\log(\ell a b+b y)}-\frac{2}{\log 2}+\int_2^{\ell a b+b y}\frac{\d t}{(\log t)^{2}}.
 	\end{align}
 	Note that
 	$$
 	\frac{\ell a b+b y}{\log(\ell a b+b y)}=\frac{\ell a b+b y}{\log g}\frac{\log g}{\log(\ell a b+b y)}
 	$$ 
 	and 
 	$$
 	\frac{\log g}{\log(\ell a b+b y)}=\frac{\log b+\log\bigl((\ell+1)a-\frac{a}{b}-1\bigr)}{\log b+\log(\ell a+y)}\to 1\quad(b\to\infty)\,. 
 	$$ 
 	Hence, we have 
 	\begin{align}\label{thm-1-eq-2}
 		\frac{\ell ab+by}{\log(\ell ab+by)} \sim \frac{\ell ab+by}{\log g} \quad (b\to \infty)\,.
 	\end{align}
 	It is clear that
 	\begin{align*}
 		\int_2^{\ell a b+b y}\frac{\d t}{(\log t)^{2}}&=\int_2^{\sqrt{b}}\frac{\d t}{(\log t)^{2}}+\int_{\sqrt{b}}^{\ell a b+b y}\frac{\d t}{(\log t)^{2}}\\
 		&<\frac{\sqrt{b}}{(\log 2)^{2}}+\frac{\ell a b+b y}{(\log \sqrt{b})^{2}}\\
 		&<\frac{\sqrt{b}}{(\log 2)^{2}}+\frac{\ell ab+by}{\log g}\frac{8\log g}{(\log b)^{2}}
 	\end{align*} 
 	and 
 	$$
 	\frac{ 8  \log g}{(\log b)^{2}}=\frac{8\log b + 8\log \big((\ell+1)a-\frac{a}{b}-1\big)}{(\log b)^{2}}\to 0\quad(b\to\infty)\,. 
 	$$ 
 	So, we have 
 	\begin{align}\label{thm-1-eq-3}
 		\int_2^{\ell a b+b y}\frac{\d t}{(\log t)^{2}} = \frac{\ell ab+by}{\log g}o(1) \quad (b\to \infty)\,.
 	\end{align}
 	Thus, by \eqref{thm-1-eq-1}, \eqref{thm-1-eq-2}, and \eqref{thm-1-eq-3} we have
 	$$
 	\int_2^{\ell a b+b y}\frac{\d t}{\log t}=\frac{\ell a b+b y}{\log g}\bigl(1+o(1)\bigr)\quad(b\to\infty)\,. 
 	$$ 
 	Hence, 
 	\begin{align*}
 		\frac{1}{\varphi(a)}\sum_{\substack{y=1\\\gcd(y,a)=1}}^{a-1}\int_2^{\ell a b+b y}\frac{\d t}{\log t}
 		&=\frac{1+o(1)}{\varphi(a)\log g}\sum_{\substack{y=1\\\gcd(y,a)=1}}^{a-1}(\ell a b+b y)\\
 		&=\frac{1+o(1)}{\varphi(a)\log g}\left(\ell a b\varphi(a)+b\cdot\frac{a\varphi(a)}{2}\right)\\
 		&=\frac{1+o(1)}{\log g}\left(\ell a b+\frac{a b}{2}\right)\,. 
 	\end{align*} 
 	Therefore,  
 	\begin{align*}
 		\pi_{\ell,a,b}&=\int_2^{g}\frac{\d t}{\log t}-\frac{1+o(1)}{\log g}\frac{(2\ell+1)a b}{2}+O(g e^{-D\sqrt{\log g}})\\
 		&=\left(1-\frac{(2\ell+1)a b}{2(\ell a b+a b-b)}+o(1)\right)\frac{g}{\log g}\\
 		&=\left(1-\frac{(2\ell+1)a}{2(\ell a+a-1)} + o(1)  \right)\pi(g)\,. 
 	\end{align*} 
 	This completes the proof of Theorem \ref{th1}.
 \end{proof}
 
 \subsection{Proof of Theorem \ref{thm-1}}
\begin{proof}[Proof of Theorem \ref{thm-1}]
	Let $c= 1-\frac{1}{8(\ell+1)}$. Then 
	\begin{align*}
		cg
		=(\ell+1)ab-a-b-\frac{(\ell+1) ab-a-b}{8(\ell+1)}
		>\ell ab +\left(\frac{7}{8}-\frac{1}{a}-\frac{1}{b}\right)ab>\ell ab
	\end{align*}
	by $\ell\geq 1$ and $b>a>e^{\ell+1}>7.$
	Since $cg>\ell ab$, we have
	\begin{align}\label{ineq-1}
		\pi_{\ell,a,b}
		&=\#\left\{p\in \cP: \ell ab<p\leq g, p=\ell ab+ax+by,(x,y)\in\Z_{\geq 0}^{2}\right\}\nonumber\\
		&\geq \pi(g)-\pi(cg)-\#\{p\in\cP: cg< p\leq g, p\notin S\}
	\end{align}
	by Lemma \ref{lem-5-1}. The arguments will be separated into the following seven cases.

	\textbf{Case 1.} $a\geq \max\{e^{(\ell+1)},6 \cdot 10^{4}+1\}$.
	By $a>6 \cdot 10^{4}$ and (\ref{C1-ineq-1}), we have
	\begin{align}\label{C1-ineq-2}
		\frac{1}{4(\ell+1)} \left( \frac{1}{8}+\frac{1}{a}+\frac{2^{\omega(a)}}{\varphi(a)} \right) 
		&< \frac{1}{4(\ell+1)}\left(\frac{1}{8}+\frac{1}{6\cdot 10^{4}}+0.00556112 \right) \nonumber\\
		&= 0.03264444\ldots \cdot \frac{1}{\ell+1},
	\end{align}
	and
	\begin{equation*}
		g=(\ell+1)ab-a-b>ab.
	\end{equation*}
	It follows that
	\begin{align}\label{thm-2-C1-ineq-3}
		1-\frac{\log \big( 8(\ell+1)a \big) }{\log g}
		&\geq 1-\frac{\log \big( 8(\ell+1)a \big)}{\log (ab)}\nonumber\\
		&=1-\frac{\log 8+\log (\ell+1)+\log a}{\log a+\log b }.
	\end{align}
	If $\ell\geq \lceil \log (6\cdot 10^{4}+1)-1 \rceil=11$, that is, $\max\{e^{(\ell+1)},6\cdot 10^{4}+1\}=e^{(\ell+1)}$, then it follows that
	\begin{align}\label{C1-ineq-1-2}
		\frac{\log 8+\log (\ell+1)+\log a}{\log a+\log b}
		&=\frac{\log 8/\log a+\log(\ell+1)/\log a+1}{1+\log b/\log a}\nonumber\\
		&\leq \frac{\log 8/\log e^{(\ell+1)}+\log (\ell+1)/\log e^{(\ell+1)}+1}{1+\log b/\log a}\nonumber\\
		&\leq \frac{\log 8/12+\log 12/12+1}{2}\nonumber\\
		&= 0.69018117 \ldots .
	\end{align}
	If $1\leq \ell< \lceil \log (6\cdot 10^{4}+1)-1 \rceil=11$, that is, $\max\{e^{(\ell+1)},6\cdot 10^{4}+1\}=6\cdot 10^{4}+1$, then it follows that
	\begin{align}\label{C1-ineq-1-3}
		\frac{\log 8+\log (\ell+1)+\log a}{\log a+\log b}
		&=\frac{\log 8/\log a+\log(\ell+1)/\log a+1}{1+\log b/\log a}\nonumber\\
		&\leq \frac{\log 8/\log (6\cdot 10^{4})+\log (\ell+1)/\log (6\cdot 10^{4})+1}{1+\log b/\log a}\nonumber\\
		&\leq \frac{\log 8/\log(6\cdot 10^{4})+\log 11/\log(6\cdot 10^{4})+1}{2}\nonumber\\
		&= 0.70347646 \ldots.
	\end{align}
	Hence, from (\ref{thm-2-C1-ineq-3}), (\ref{C1-ineq-1-2}), and (\ref{C1-ineq-1-3}), we conclude that
	\begin{align*}
		1-\frac{\log \big( 8(\ell+1)a \big) }{\log g} 
		&\geq 1-\frac{\log 8+\log (\ell+1)+\log a}{\log a+\log b}\\
		&\geq 1-0.703476469\\
		&=0.296523531, \quad (\ell\geq 1). 
	\end{align*}
	It follows that
	\begin{equation}\label{C1-ineq-3}
		\left( 1-\frac{\log \big( 8(\ell+1)a \big) }{\log g}
		\right)^{-1}\leq 3.37241383.
	\end{equation}
	Hence, by Lemma \ref{lem-6}, (\ref{C1-ineq-2}), and (\ref{C1-ineq-3}), we have
	\begin{align}\label{C1-ineq-5}
		&\#\left\{p\in\cP:cg< p\leq g, p\notin S \right\} \nonumber\\
		< & \frac{1}{4(\ell+1)} \left( \frac{1}{8}+\frac{1}{a}+\frac{2^{\omega(a)}}{\varphi(a)} \right) \cdot \left( 1-\frac{\log \big(8(\ell+1)a \big)}{\log g}\right)^{-1}\cdot \frac{g}{\log g}\nonumber\\
		<& (0.032644449 \cdot 3.37241383)\cdot\frac{1}{\ell+1} \frac{g}{\log g}\nonumber\\
		<& 0.110090592 \cdot \frac{1}{\ell+1} \frac{g}{\log g}.
	\end{align}
	It is clear that $g=(\ell+1)ab-a-b>\max\{\ell e^{2(\ell+1)},7\cdot 10^{5}\}$ by $a\geq \max\{e^{\ell+1},6\cdot 10^{4}+1\}$.
	By (\ref{ineq-1}), Lemma \ref{lem-2}, and (\ref{C1-ineq-5}), we have
	\begin{align*}
		\pi_{\ell,a,b}&\geq \pi(g)-\pi(cg)-\#\{p\in\cP: cg< p\leq g, p\notin S\}\\
		&\geq  (0.11604925-0.110090592) \cdot \frac{1}{\ell+1}\frac{g}{\log g}\\
		&>0.0059 \cdot \frac{1}{\ell+1}\frac{g}{\log g}.
	\end{align*}
	
	\textbf{Case 2.} $\max\{ e^{\ell+1},631\}\leq a\leq 6\cdot 10^{4}$. By $\ell\geq 1$ and $\max\{ e^{\ell+1},631\}\leq a<b$, we have 
	\begin{align}\label{C2-ineq-3}
		g&=(\ell+1)ab-a-b\nonumber\\
		&=\ell ab+(a-1)(b-1)-1\nonumber\\
		& \geq \ell a(a+1)+a(a-1)-1\nonumber\\
		&= \ell a(a+1)+(a-1)^{2}+a-2 \nonumber\\
		&> \max\{\ell e^{2(\ell+1)},7\cdot 10^{5}\}.
	\end{align} 
	Then,
	\begin{equation}\label{C2-ineq-1}
		\frac{\log \big(8(\ell+1)a\big)}{\log g}
		<\frac{\log\big(8(\ell+1)a\big)}{\log \big( \ell a(a+1)+(a-1)^{2}+a-2 \big)}<1.
	\end{equation}
	By Lemma \ref{lem-6}, (\ref{C2-ineq-1}), and the help of a computer\footnote{In this case, we employed a Python program--utilizing the Python packages numpy, math, sympy, and numba--to iterate over all integers $\ell$ and $a$, to find the maximum value of the coefficient of $\frac{1}{\ell+1}\frac{g}{\log g}$. For other analogous cases, we also used Python for computational purposes.}, 
	we have 
	\begin{align}
		&\#\left\{p\in\cP:cg< p\leq g, p\notin S \right\}\nonumber\\
		<& \frac{1}{4(\ell+1)\varphi(a)} \Bigg(\sum_{\stackrel{1\leq y\leq a/8+1}{\gcd(y,a)=1} }1\Bigg) \left( 1-\frac{\log \big(8(\ell+1)a\big)}{\log g}\right)^{-1}\cdot \frac{g}{\log g}\nonumber\\
		<& \frac{1}{4(\ell+1)\varphi(a)} \Bigg(\sum_{\stackrel{1\leq y\leq a/8+1}{\gcd(y,a)=1} }1\Bigg) \left( 1-\frac{\log\big(8(\ell+1)a\big)}{\log \left( \ell a(a+1)+(a-1)^{2}+a-2 \right)} \right)^{-1}\cdot \frac{g}{\log g}\nonumber\\
		\leq& 	0.10985491\ldots\cdot \frac{1}{\ell+1}\frac{g}{\log g}, \label{thm-2-C2-ineq-1}
	\end{align}
	and the last equality holds if $\ell=5$ and $a=660$.
	By \eqref{ineq-1}, (\ref{C2-ineq-3}), Lemma \ref{lem-2}, and (\ref{thm-2-C2-ineq-1}), 
	we have
	\begin{align*}
		\pi_{\ell,a,b}
		&\geq \pi(g)-\pi(cg)-\#\{p\in\cP: cg< p\leq g, p\notin S\}\\
		&> (0.11604925- 0.10985492)\cdot \frac{1}{\ell+1}\frac{g}{\log g}\\
		&> 0.0061 \cdot \frac{1}{\ell+1}\frac{g}{\log g}.
	\end{align*}
	
	\textbf{Case 3.} $\max\{e^{\ell+1},149\}\leq a\leq 630$ and $g\geq 3\cdot 10^{6}$. It is clear that 
	$$
	1\leq\ell\leq \lfloor \log 630 \rfloor-1=5
	$$
	and 
	$$
	g=\ell ab+(a-1)(b-1)-1>\ell ab> \ell e^{2(\ell+1)}.
	$$
	It follows that $g\geq \max\left\{\ell e^{2(\ell+1)}, 3\cdot 10^{6}\right\}.$
	By $g=(\ell+1) ab-a-b\geq 3\cdot 10^{6}$ and $b>a$, we have 
	$$
	b\geq \max\left\{\left\lceil \frac{3\cdot 10^{6}+a}{(\ell+1)a-1} \right\rceil, a+1 \right\}
	:=\delta_{\ell,a}.
	$$
	It follows that
	\begin{align}
		g&= \ell ab+(a-1)(b-1)-1\nonumber\\
		&\geq\ell a\delta_{\ell,a} +(a-1)(\delta_{\ell,a}-1)-1.
	\end{align}
	Then,
	\begin{equation}\label{C3-ineq-1}
		\frac{\log \big(8(\ell+1)a\big)}{\log g}
		\leq \frac{\log\big(8(\ell+1)a\big)}{\log \big(\ell a\delta_{\ell,a} +(a-1)(\delta_{\ell,a}-1)-1 \big)}<1.
	\end{equation}
	By Lemma \ref{lem-6}, (\ref{C3-ineq-1}), and the help of a computer, we have 
	\begin{align}\label{thm-2-C3-ineq-1}
		&\#\left\{p\in\cP:cg< p\leq g, p\notin S \right\}\nonumber\\
		<& \frac{1}{4(\ell+1)\varphi(a)} \Bigg(\sum_{\stackrel{1\leq y\leq a/8+1}{\gcd(y,a)=1} }1\Bigg) \left( 1-\frac{\log \big(8(\ell+1)a\big)}{\log g}\right)^{-1}\cdot \frac{g}{\log g}\nonumber\\
		\leq & \frac{1}{4(\ell+1)\varphi(a)} \Bigg(\sum_{\stackrel{1\leq y\leq a/8+1}{\gcd(y,a)=1} }1\Bigg) \left( 1-\frac{\log\big(8(\ell+1)a\big)}{\log \left(\ell a\delta_{\ell,a} +(a-1)(\delta_{\ell,a}-1)-1 \right)} \right)^{-1}\cdot \frac{g}{\log g}\nonumber\\
		\leq & 0.10537860\ldots\cdot \frac{1}{\ell+1}\frac{g}{\log g},
	\end{align}
		and the last equality holds if $\ell=5$ and $a=630$.
	By (\ref{ineq-1}), $g \geq \max\left\{\ell e^{2(\ell+1)},3\cdot 10^{6}\right\}$, Lemma \ref{lem-2}, and (\ref{thm-2-C3-ineq-1}), we have
	\begin{align*}
		\pi_{\ell,a,b}
		&\geq \pi(g)-\pi(cg)-\#\left\{p\in\mathcal{P}: cg< p\leq g, p\notin S\right\}\\
		&> (0.11604925- 0.10537861)\cdot \frac{1}{\ell+1}\frac{g}{\log g}\\
		&> 0.0106 \cdot \frac{1}{\ell+1}\frac{g}{\log g}.
	\end{align*}
	
	\textbf{Case 4.} $\max\{e^{\ell+1},16\}\leq a\leq 148$ and $g\geq 7\cdot 10^{5}$.
	By $g=(\ell+1) ab-a-b\geq 7\cdot 10^{5}$ and $b>a$, we have 
	$$
	b\geq \max\left\{\left\lceil \frac{7\cdot 10^{5}+a}{(\ell+1)a-1} \right\rceil, a+1 \right\}
	:=\delta_{\ell,a}.
	$$
	By Lemma \ref{lem-6}, (\ref{C3-ineq-1}), and the help of a computer, we have 
	\begin{align}
		&\#\left\{p\in\cP:cg< p\leq g, p\notin S \right\}\nonumber\\
		<& \frac{1}{4(\ell+1)\varphi(a)} \Bigg(\sum_{\stackrel{1\leq y\leq a/8+1}{\gcd(y,a)=1} }1\Bigg) \left( 1-\frac{\log \big(8(\ell+1)a\big)}{\log g}\right)^{-1}\cdot \frac{g}{\log g}\nonumber\\
		\leq & \frac{1}{4(\ell+1)\varphi(a)} \Bigg(\sum_{\stackrel{1\leq y\leq a/8+1}{\gcd(y,a)=1} }1\Bigg) \left( 1-\frac{\log\big(8(\ell+1)a\big)}{\log \left(\ell a\delta_{\ell,a} +(a-1)(\delta_{\ell,a}-1)-1 \right)} \right)^{-1}\cdot \frac{g}{\log g}\nonumber\\
		\leq & 0.10937827\ldots\cdot \frac{1}{\ell+1}\frac{g}{\log g}, \label{thm-2-C4-ineq-1}
	\end{align}
	and the last equality holds if $\ell=1$ and $a=20$.
	By (\ref{ineq-1}), $g \geq \max\left\{\ell e^{2(\ell+1)}, 7\cdot 10^{5}\right\}$, Lemma \ref{lem-2}, and (\ref{thm-2-C4-ineq-1}), we have
	\begin{align*}
		\pi_{\ell,a,b}
		&\geq \pi(g)-\pi(cg)-\#\{p\in\mathcal{P}: cg< p\leq g, p\notin S\}\\
		&> (0.11604925- 0.10937827)\cdot \frac{1}{\ell+1}\frac{g}{\log g}\\
		&> 0.0066 \cdot \frac{1}{\ell+1}\frac{g}{\log g}.
	\end{align*}
	
	\textbf{Case 5.} $\max\left\{e^{\ell+1},149\right\}\leq a\leq 630$ with $g<3\cdot 10^{6}$, and $\max\left\{e^{\ell+1},16\right\}\leq a\leq 148$ with $g<7\cdot 10^{5}$. By Lemma \ref{lem-5-1} and the help of a computer\footnote{We employed the following computational method. First, we generated all primes up to $3 \cdot 10^6$ using the pyprimesieve Python package. For a given set of integer parameters $(\ell, a, b)$, we then iterated over all primes (until we complete our aim) within the interval $[\ell ab, g_{\ell, a,b}]$ in descending order. For each prime $p$ in this range, we calculated the integer $y\in [1,a-1]$ satisfying $b^{-1} p \equiv y \pmod{a}$ where $b^{-1}\in[1,a-1]$ is the multiplicative inverse of $b$ modulo $a$, and then checked the inequality $p > \ell ab + b y$. This search was implemented in Python with the numba and numpy Python packages. For an analogous case, we also do it.}, we have 
	\begin{align*}
		\pi_{\ell,a,b}
		&= \#\left\{p\in\cP : p\leq g,p=\ell ab+ax+by,(x,y)\in \Z_{\geq 0}^{2}\right\}\nonumber\\
		&= \#\left\{p\in \mathcal{P}: \ell ab< p\leq g,p\equiv by \pmod{a}, p>\ell ab+by, 1\leq y\leq a-1\right\}\nonumber\\
		&\geq 0.006 \cdot \frac{1}{\ell+1}\frac{g}{\log g}.
	\end{align*}

	\textbf{Case 6.} $e^{\ell+1}\leq a\leq 15$, that is, $\ell=1$, and $b\geq 50a^{2}$. It follows that $\ell ab+b>50a^{3}$ and
	$$
	g=\ell ab+(a-1)(b-1)-1> \ell ab> 50 a^{2}.
	$$
 By $a\leq 15$, $\min\left\{g, \ell ab+b\right\}>50a^{2}$, Lemma \ref{lem-5-1}, and Lemma \ref{lem-7}, we have
	\begin{align*}
		\pi_{\ell,a,b}&=\#\{p\in\cP :\ell ab\leq p\leq g,p=\ell ab+ax+by,(x,y)\in \Z_{\geq 0}^{2}\}\\
		&= \#\{p\in \mathcal{P}: \ell ab< p\leq g,p\equiv by \pmod{a}, p>\ell ab+by, 1\leq y\leq a-1\}\nonumber\\
		&\geq \pi(g;a,b)-\pi(\ell ab+b;a,b)\\
		&> \frac{g}{\varphi(a)\log g}-\frac{\ell ab+b}{\varphi(a)\log\big(\ell ab+b \big)}\left(1+\frac{5}{2\log\big(\ell ab+b \big)}\right)\\
		&=\frac{1}{\varphi(a)}\frac{g}{\log g}\left(1-\frac{\ell ab+b}{g}\frac{\log g}{\log\big(\ell ab+b \big)} \left(1+\frac{5}{2\log\big(\ell ab+b \big)}\right) \right).
	\end{align*}
	Since 
	\begin{equation*}
		\frac{\ell ab+b}{g}=\frac{\ell ab+b}{(\ell+1)ab-a-b}
		=\frac{\ell+\frac{1}{a}}{(\ell+1)-\frac{1}{b}-\frac{1}{a}}
		\leq \frac{\ell+\frac{1}{a}}{(\ell+1)-\frac{1}{50a^{2}}-\frac{1}{a}},
	\end{equation*}
	\begin{equation*}
		\frac{\log g}{\log(\ell ab+b)}
		< \frac{\log \big( (\ell+1)ab \big)}{\log(\ell ab)}
		=\frac{\log(1+\frac{1}{\ell})+\log (\ell ab)}{\log (\ell ab)}
		\leq 1+\frac{\log(1+\frac{1}{\ell})}{\log (50 \cdot \ell a^{3})},
	\end{equation*}
	and
	\begin{equation*}
		1+\frac{5}{2\log\big(\ell ab+b\big)}\leq 1+\frac{5}{2\log\big(50a^{2}(\ell a+1)\big)},
	\end{equation*}
	we have
	\begin{align*}
		&\quad \pi_{\ell,a,b}\cdot\frac{\log g}{g}\\
		&>\frac{1}{\varphi(a)}
		\left(1-\frac{\ell+\frac{1}{a}}{(\ell+1)-\frac{1}{50a^{2}}-\frac{1}{a}}
		\left(1+\frac{\log(1+\frac{1}{\ell})}{\log (50 \cdot \ell a^{3})} \right) 
		\left(1+\frac{5}{2\log\big(50a^{2}(\ell a+1) \big)}\right) \right)\\
		&\geq 0.04658379 \ldots \cdot \frac{1}{\ell+1}
	\end{align*}
	with the help of a computer, and the last equality holds if $a=13$.
	
	\textbf{Case 7.} $e^{\ell+1}\leq a\leq 15$, that is $\ell=1$, and $b< 50a^{2}$ with $\gcd(a,b)=1$.
	By Lemma \ref{lem-5-1} and the help of a computer, we have 
	\begin{align*}
		\pi_{\ell,a,b}
		&= \#\{p\in\cP :\ell ab< p\leq g,p=\ell ab+ax+by,(x,y)\in \Z_{\geq 0}^{2}\}\\
		&= \#\{p\in \mathcal{P}: \ell ab< p\leq g,p\equiv by \pmod{a}, p>\ell ab+by, 1\leq y\leq a-1\}\\
		&\geq 0.1 \cdot \frac{1}{\ell+1}\frac{g}{\log g}.
	\end{align*}
	
	This completes the proof of Theorem \ref{thm-1}.
\end{proof}

\subsection{Proo of Theorem \ref{th:22}}  
\begin{proof}[Proof of Theorem \ref{th:22}]
Since the $\ell$-Ap\'ery set ${\rm Ap}_\ell(a,b)=\{m_0^{(\ell)},m_1^{(\ell)},\dots,m_{a-1}^{(\ell)}\}$, satisfying
$$
{\rm (i)}\,m_i^{(\ell)}\equiv i\pmod{a}\quad{\rm (ii)}\,m_i^{(\ell)}\in S_{\ell}(a,b)\quad {\rm (iii)}\,m_i^{(\ell)}-a\in G_{\ell}(a,b),  
$$ 
is given by ${\rm Ap}_\ell(a,b)=\{b(\ell a+v):0\le v\le a-1\}$ (\cite[pp.58-59]{Kom2}),
we have 
\begin{align*}  
	\pi^{*}_{\ell,a,b}
	&=\sum_{v=0}^{a-1} \#\left\{p\in\mathcal P: p<b(\ell a+v),\,p\equiv b v\pmod{a}\right\}\\
	&\ge \pi(\ell ab-1;a,0)+\sum_{\substack{v=1\\\gcd(v,a)=1}}^{a-2}\pi\bigl(b(\ell a+v);a,b v\bigr)+\pi\bigl(g;a,b(a-1)\bigr)\,. 
\end{align*} 
By Lemma \ref{pi(x;m,l)}, there is $X(m)$ such that for $x\geq X(m)$, then
$$
\pi(x;m,l)>\frac{x}{\varphi(m)\log x}.
$$
Thus, if $\ell ab\ge X(a)$, we have 
\begin{align*}  
\pi^{*}_{\ell,a,b}
&\ge 0+\sum_{\substack{v=1\\\gcd(v,a)=1}}^{a-2}\frac{1}{\varphi(a)}\frac{b(\ell a+v)}{\log\bigl(b(\ell a+v)\bigr)} +\frac{1}{\varphi(a)}\frac{g}{\log g}\\
&\ge \sum_{\substack{v=1\\\gcd(v,a)=1}}^{a-2}\frac{1}{\varphi(a)}\frac{b(\ell a+v)}{\log g}+\frac{1}{\varphi(a)}\frac{g}{\log g}\\ 
&=\frac{b}{\varphi(a)\log g}\sum_{\substack{v= 1\\ \gcd(v,a)=1}}^{a-1}(\ell a+v)-\frac{a}{\varphi(a)\log g}\,.
\end{align*}
Since 
$$
\sum_{\substack{v=1 \\\gcd(v,a)=1}}^{a-1}v=\frac{a\varphi(a)}{2}\,, 
$$ 
we have 
\begin{align*}  
	\pi^{*}_{\ell,a,b}
	&\ge\frac{b}{\varphi(a)\log g}\left(\ell a \varphi(a)+\frac{a\varphi(a)}{2}\right)-\frac{a}{\varphi(a)\log g}\\ 
	&=\frac{(2\ell+1)ab}{2\log g}-\frac{a}{\varphi(a)\log g}\\
	&=\frac{(2\ell+1)ab}{2(\ell a+a-1)b}\frac{g+a}{\log g}-\frac{a}{\varphi(a)\log g}\\
	&=\frac{(2\ell+1)a}{2(\ell a+a-1)}\frac{g}{\log g}+\frac{(2\ell+1)a}{2(\ell a+a-1)} \frac{a}{\log g}-\frac{a}{\varphi(a)\log g}\\
	&=\frac{(2\ell+1)a}{2(\ell a+a-1)}\frac{g}{\log g}+\left( \frac{2\ell+1}{2(\ell+1-1/a)}-\frac{1}{\varphi(a)} \right)\frac{a}{\log g}\\
	&>\frac{(2\ell+1)a}{2(\ell a+a-1)}\frac{g}{\log g}.
\end{align*} 

This completes the proof of Theorem \ref{th:22}.
\end{proof}
\subsection{Proof of Theorem \ref{thm-4}}
\begin{proof}[Proof of Theorem \ref{thm-4}]
Assuming that $a=2$ and $b> 2$ with $\gcd(a,b)=1$, that is, $b>2$ is an odd integer. Let $p$ be a prime integer such that $p\in S_{1}(2,b)$ and $p\leq g_{1,2,b}$. By Lemma \ref{lem-5-1}, there exists a non-negative $x$ such that $p=2b+2x+b$ with $1\leq x\leq b-1$. However, this leads to a contradiction because $p\leq g_{1,2,b}=4b-2-b$. Subsequently, we conclude that $\pi_{1,2,b}^{*}=\pi(g_{1,2,b})$. 
If $g_{1,2,b}\geq 17$, by Lemma \ref{pi(x)estimate}, we have
$$
\pi_{1,2,b}^{*}=\pi(g_{1,2,b})>\frac{g_{1,2,b}}{\log g_{1,2,b}}.
$$
If $g_{1,2,b}=3b-2<17$, then $b=3$ or $b=5$. In these cases, we have
\begin{align*}
	\pi^{*}_{1,2,3}=\#\{2,3,5,7\}>\frac{7}{\log 7}=\frac{g_{1,2,3}}{\log g_{1,2,3}},
\end{align*}
and
\begin{align*}
	\pi^{*}_{1,2,5}=\#\{2,3,5,7,11,13\}>\frac{13}{\log 13}=\frac{g_{1,2,5}}{\log g_{1,2,5}}.
\end{align*}
 Hence, Theorem \ref{thm-4} holds for $a=2$, and we will assume that $a\geq 3$.

Suppose that $\vartheta<1$ is a positive number and $\theta=(\ell+\vartheta)/(\ell+1)$ such that $\theta g>\ell ab+a$. It follows that 
$$ 
a<\theta g-\ell ab = \vartheta ab -\theta a-\theta b<\vartheta ab
$$ 
and $0<\vartheta -\theta/b-\theta/a<\vartheta<1$.
By Lemma \ref{lem-5-1} we have
\begin{align}\label{thm-4-ineq-7}
	\pi^{*}_{\ell,a,b} 
	&\geq \pi(\theta g)- \#\{p\in\cP: \ell ab < p \leq \theta g, p\in S\}\nonumber\\
	&= \pi(\theta g)-\#\{p\in\cP: \ell ab<p\leq \theta g, p=\ell ab+ax+by,(x,y)\in\Z_{\geq 0}^{2}\}.
\end{align}
If $y=0$ and $p=\ell ab+ax+by$ for some non-integer $x$, then we have $p=a(\ell b+x)$ and hence $p$ is not a prime integer by $\ell\geq 1$ and $b>a>2$. 
So, we have
\begin{align}\label{thm-4-ineq-1}
	&\#\{p\in\cP: \ell ab<p\leq \theta g, p=\ell ab+ax+by,(x,y)\in\Z_{\geq 0}^{2}\}\nonumber\\
	= &\#\{p\in\cP: \ell ab<p\leq \theta g, p=\ell ab+ax+by,(x,y)\in\Z_{\geq 0}^{2}, y\geq 1\}\nonumber\\
	\leq &\#\{p\in\cP: \ell ab<p\leq \theta g, p\equiv by \pmod{a},1\leq by\leq \theta g-\ell ab\}\nonumber\\
	=  &\sum_{\stackrel{1\leq y\leq (\theta g-\ell ab)/b}{\gcd(y,a)=1}}\big( \pi(\theta g;a,by)- \pi(\ell ab;a,by)\big).
\end{align}
By \eqref{thm-4-ineq-7}, \eqref{thm-4-ineq-1}, and $\theta g-\ell ab<\vartheta ab$, we have
\begin{align}\label{thm-4-ineq-3}
		\pi^{*}_{\ell,a,b} \geq \pi(\theta g)-\sum_{\stackrel{1\leq y\leq \vartheta a}{\gcd(y,a)=1}} \big( \pi(\theta g;a,by)- \pi(\ell ab;a,by)\big).
\end{align}
By Lemma \ref{lem-3} and $\theta g-\ell ab>a$, we have
\begin{align}\label{thm-4-ineq-4}
	\pi(\theta g;a,by)- \pi(\ell ab;a,by) 
	&< \frac{2(\theta g-\ell ab)}{\varphi(a)\log \big( (\theta g-\ell ab)/a \big)}\nonumber\\
	&=\frac{2}{\varphi(a)} \frac{\log  g}{\log\big((\theta g-\ell ab)/a \big)} \frac{\theta g-\ell ab}{g} \frac{g}{\log  g}.
\end{align} Recalling that $g=(\ell+1)b-a-b$ and $\theta= (\ell+\vartheta)/(\ell+1)$, we have
\begin{align}\label{thm-4-ineq-5}
	\frac{\theta g-\ell ab}{g}=\frac{\vartheta ab-\theta a-\theta b}{(\ell+1)ab-a-b}=\frac{\vartheta-\theta/b-\theta/a}{\ell+1-1/b-1/a}, 
\end{align}
and hence
\begin{align}
	\frac{\log\big((\theta g-\ell ab)/a \big)}{\log g}
	&= 1+\frac{\log\big((\theta-\ell ab)/ g\big)}{\log g}-\frac{\log a}{\log g}\nonumber\\
	&= 1-\frac{\log a}{\log  g}+\frac{\log( \vartheta-\theta/b-\theta/a)}{\log  g}-\frac{\log(\ell+1-1/b-1/a)}{\log  g}\nonumber\\
	&> 1-\frac{\log a}{\log  g}-\frac{\log(\ell+1)}{\log g}-\frac{\log( \vartheta-\theta/a-\theta/b)^{-1}}{\log  g} \label{thm-4-ineq-6}\\
	&:= H(\ell,a,b,\vartheta).\nonumber
\end{align}
If $\theta g>17$, by Lemma \ref{pi(x)estimate}, we have 
\begin{align}\label{thm-4-ineq-4-1}
	\pi(\theta g)>\frac{\theta g}{\log (\theta g)}>\theta \frac{g}{\log g}.
\end{align}
Therefore, by \eqref{thm-4-ineq-3}, \eqref{thm-4-ineq-4}, \eqref{thm-4-ineq-5}, \eqref{thm-4-ineq-6}, and \eqref{thm-4-ineq-4-1}, we have
\begin{align}
	\pi^{*}_{\ell,a,b} 
	&\geq \pi(\theta g)-\sum_{\stackrel{1\leq y\leq \vartheta a}{\gcd(y,a)=1}} \big( \pi(\theta g;a,by)- \pi(\ell ab;a,by)\big)\nonumber\\
	&> \left( \theta -\frac{2}{\varphi(a)}
		\Bigg(\sum_{\stackrel{1\leq y\leq \vartheta a}{\gcd(y,a)=1}}1 \Bigg)
		\frac{\log  g}{\log\big((\theta g-\ell ab)/a \big)} \frac{\theta g-\ell ab}{g}
	\right) \frac{g}{\log g}\nonumber\\
	&> \left( \theta -\frac{2}{\varphi(a)}\Bigg(\sum_{\stackrel{1\leq y\leq \vartheta a}{\gcd(y,a)=1}}1\Bigg) H(\ell,a,b,\vartheta)^{-1} \frac{\vartheta - \theta/b -\theta/a}{\ell+1-1/b-1/a}
\right) \frac{g}{\log g}. \label{thm-4-ineq-7-2}
\end{align}
The arguments will be separated into the following four cases.

\textbf{Case 1.} $a>\max\{\ell,6\cdot 10^{4}\}$.
Let $\vartheta=0.055$ and $\theta=(\ell+\vartheta)/(\ell+1)$. By $g=(\ell+1)ab-a-b$ and $b>a>6\cdot 10^{4}$, we have 
\begin{align*}
	\theta g-\ell ab 
	=\vartheta ab-\theta a-\theta b
	>0.055ab-a-b 
	>(0.055a-2)b
	>a
\end{align*}
 and $\theta g>17$. It is easy to see that 
 \begin{equation*}
	\vartheta -\frac{\theta}{b}-\frac{\theta}{a}>\vartheta-\frac{1}{b}-\frac{1}{a}>0.055-\frac{2}{6\cdot 10^{4}}>0.05496666,
\end{equation*}
and hence $1<(\vartheta-\theta/b-\theta/a)^{-1}<18.2$.
Recalling that $b>a$, we have
\begin{align*}
	g&=(\ell+1)ab-a-b\\
	&=\ell ab+(a-1)(b-1)-1\\
	&\geq \ell a(a+1)+(a-1)a-1\\
	&=(\ell+1)a^{2}+\ell a-a-1\\
	&\geq (\ell+1)a^{2}-1.
\end{align*}
Therefore, by \eqref{thm-4-ineq-6}, $g\geq (\ell+1)a^{2}-1$, and  $1<(\vartheta-\theta/a-\theta/b<\vartheta)^{-1}<18.2$, we have
\begin{align*}
	&H(\ell,a,b,\vartheta)\\
	>& 1-\frac{\log a}{\log  \big( (\ell+1)a^{2}-1 \big)}-\frac{\log(\ell+1)}{\log  \big( (\ell+1)a^{2}-1 \big)}-\frac{\log 18.2}{\log \big( (\ell+1)a^{2}-1 \big)}\\
	=&\frac{\log  \big( (\ell+1)a^{2}-1 \big)-\log a-\log(\ell+1)}
	{\log  \big( (\ell+1)a^{2}-1 \big)}
	-\frac{\log 18.2}{\log \big( (\ell+1)a^{2}-1 \big)}\\
	>& \frac{\log\left(1-\frac{1}{(\ell+1)a^{2}}\right)+\log a}{\log\left({(\ell+1)a^{2}}\right)}-\frac{\log 18.2}{\log \big( (\ell+1)a^{2}-1 \big)}.
\end{align*}
Recalling that  $a>\ell$ and $a>6\cdot 10^{4}$, we have
\begin{align*}
	H(\ell,a,b,\vartheta)
	&> \frac{\log\left(1-\frac{1}{(\ell+1)a^{2}}\right)+\log a}{\log a^{3}}-\frac{\log18.2}{\log \big( (\ell+1)a^{2}-1 \big)}\\
	&=\frac{1}{3}+\frac{\log(1-\frac{1}{(\ell+1)a^{2}})}{3\log a}-\frac{\log 18.2}{\log \big( (\ell+1)a^{2}-1 \big)}\\
	&\geq \frac{1}{3}+ \frac{\log(1-\frac{1}{2\cdot (6\cdot 10^{4})^{2}})}{3\log (6\cdot 10^{4})}-\frac{\log 18.2}{\log \big(2\cdot (6\cdot 10^{4})^{2}-1\big)}\\
	&= 0.2055024643\ldots, 
\end{align*}
and hence
\begin{align}\label{thm-4-C1-ineq-1}
	H(\ell,a,b,\vartheta)^{-1} <4.866121719.
\end{align}
Recalling that $\vartheta=0.055$, $\theta=(\ell+\vartheta)/(\ell+1)<1$, and $b>a>6\cdot 10^{4}$, we have
\begin{align}\label{thm-4-C1-ineq-2}
\frac{\vartheta-\theta/b-\theta/a}{\ell+1-1/b-1/a}
<\frac{0.055}{\ell+1-2/(6\cdot 10^{4})}
<\frac{0.055+1.84\cdot 10^{-6}}{\ell+1} \quad(\ell\geq 1).
\end{align}
By Lemma \ref{lem-4}, $a>6\cdot 10^{4}$, and  \eqref{C1-ineq-1}, we have
\begin{align}
	\frac{1}{\varphi(a)}\sum_{\stackrel{1\leq y\leq \vartheta a}{\gcd(y,a)=1}} 1
	&\leq \frac{1}{\varphi(a)} \left(\frac{\varphi(a)}{a} \cdot \vartheta a+2^{\omega(a)}\right)\nonumber\\
	&= \vartheta+\frac{2^{\omega(a)}}{\varphi(a)}\nonumber\\
	&< 0.055+0.00556112\nonumber\\
	&= 0.06056112. \label{thm-4-C1-ineq-3}
\end{align}
Therefore, by (\ref{thm-4-ineq-7-2}), \eqref{thm-4-C1-ineq-1}, \eqref{thm-4-C1-ineq-2}, and \eqref{thm-4-C1-ineq-3}, we have
\begin{align*}\label{thm-4-C1-ineq-4}
	\pi^{*}_{\ell,a,b} 
	 &> \left( \theta-
	\frac{2}{\varphi(a)}
	\Bigg( \sum_{\stackrel{1\leq y\leq \vartheta a}{\gcd(y,a)=1}}1\Bigg) H(\ell,a,b,\vartheta)^{-1} 
	\frac{\vartheta-\theta/b-\theta/a}{\ell+1-1/b-1/a} 
	\right) \frac{g}{\log g}\\
	&>\left( 
		\frac{\ell+0.055}{\ell+1}- 2\cdot 0.06056112 \cdot 4.866121719 \cdot \frac{0.055+1.84\cdot 10^{-6}}{\ell+1}
	\right)
	\frac{g}{\log g}\\
	&> \frac{\ell+0.02258}{\ell+1}\frac{g}{\log g}.
\end{align*}

\textbf{Case 2.} $\max\{\ell,650\}<a\leq 6\cdot 10^{4}$. Let $\vartheta=0.05 $ and $\theta = (\ell+\vartheta)/(\ell+1)$. By $g=(\ell+1)ab-a-b$ and $b>a>650$, we have 
$$
\theta g-\ell ab 
= \vartheta ab-\theta a-\theta b
> 0.05ab-a-b
> (0.05a-2)b
> a
$$
and $\theta g>17$.
It is easy to see that 
\begin{equation*}
	\vartheta-\frac{\theta}{b}-\frac{\theta}{a}
	>\vartheta-\frac{2\theta}{a}
	>0.05 -\frac{2}{650}
	>0,
\end{equation*}
and hence 
\begin{align}\label{thm-4-C2-ineq-1}
	\left(\vartheta-\frac{2\theta}{a}\right)^{-1}
	>\left(\vartheta-\frac{\theta}{b}-\frac{\theta}{a}\right)^{-1}>1.
\end{align}
By \eqref{thm-4-ineq-6}, $g\geq (\ell+1)a^{2}-1$, and \eqref{thm-4-C2-ineq-1}, we have 
\begin{align*}
		H(\ell,a,b,\vartheta)
	&= 1-\frac{\log a}{\log  g}-\frac{\log(\ell+1)}{\log g}-\frac{\log( \vartheta-\theta/a-\theta/b)^{-1}}{\log  g}\nonumber\\
	&> 1-\frac{\log\big( a(\ell+1) \big)}{\log\big( (\ell+1)a^{2}-1 \big)}-\frac{\log(\vartheta -2\theta/a)^{-1}}{\log \big( (\ell+1)a^{2}-1\big)}
	\nonumber\\
	&:= H_{2}(\ell,a,\vartheta)
\end{align*}
and hence $H(\ell,a,b,\vartheta)^{-1}< H_{2}(\ell,a,\vartheta)^{-1}$.
Recalling that $\vartheta=0.05$ and $650<a<b$, we have
\begin{align}\label{thm-4-C2-ineq-2}
	\frac{\vartheta-\theta/b-\theta/a}{\ell+1-1/b-1/a}
	<
	\frac{\vartheta-\theta/a}{\ell+1-2/a}.
\end{align}
Hence, by \eqref{thm-4-ineq-7-2}, $H(\ell,a,b,\vartheta)^{-1}<H_{2}(\ell,a,\vartheta)^{-1}$, \eqref{thm-4-C2-ineq-2}, and the help of a computer, we conclude that
\begin{align*}
	\pi^{*}_{\ell,a,b} 
	&>\left( 
		\frac{\ell+\vartheta}{\ell+1}
		-\frac{2}{\varphi(a)}
		\Bigg(
			\sum_{\stackrel{1\leq y\leq \vartheta a}{
				\gcd(y,a)=1}} 1
		\Bigg)
			H_{2}(\ell,a,\vartheta)^{-1}
		\frac{\vartheta-\theta/a}{\ell+1-2/a}
	\right)  \frac{g}{\log g}\\
	&\geq \frac{\ell+0.02088778\ldots}{\ell+1}\frac{g}{\log g},
\end{align*}
and the last equality holds if $\ell=669$ and $a = 670$.

\textbf{Case 3.} $\max\{\ell,15\}<a\leq 650$. If $b\leq 1000$ with $\gcd(a,b)=1$, we have
\begin{align*}
	\pi^{*}_{\ell,a,b}
	&=\pi(\ell ab)+\#\{p\in\cP: \ell ab<p\leq g,p\notin S\}\\
	&= \pi(\ell ab)+\#\{p\in\cP:\ell ab<p\leq g, g-p =ax+by,(x,y)\in\Z_{\geq 0}^{2}\}\\
	&= \pi(\ell ab)+\#\left\{ p\in\cP: \ell ab<p\leq g,g-p \equiv by \pmod{a}, g-p\geq by,0\leq y\leq a-1 \right\}\\
	&>\frac{\ell ab}{\log (\ell ab)}+\frac{0.021}{\ell+1}\frac{g}{\log g}\\
	&>\frac{\ell+0.021}{\ell+1}\frac{g}{\log g}
\end{align*}
by Lemma \ref{lem-5}, $g>\ell a b>17$, Lemma \ref{pi(x)estimate}, and the help of a computer.
 If $b\geq 1001$ with $\gcd(a,b)=1$, we have
 \begin{align*}
 	g&=(\ell+1)ab-a-b\\
 	&=\ell ab+(a-1)(b-1)-1\\
 	&\geq 1001 \ell a+1000(a-1)-1.
 \end{align*}
 Let $\vartheta= 0.066$ and $\theta=(\ell+\vartheta)/(\ell+1)$. Then we have 
  $$
 \vartheta -\frac{\theta}{b} -\frac{\theta}{a}
 >\vartheta-\frac{1}{b}-\frac{1}{a}
 >0.066-\frac{1}{16}-\frac{1}{1001}>0.002,
 $$
 $$\theta g-\ell ab
 =\vartheta ab-\theta a-\theta b
 >\vartheta ab-a-b
 =\left( \vartheta-\frac{1}{b}-\frac{1}{a} \right) ba
 >a,
 $$
 and 
 $\theta g>17$ by $15<a\leq 650$ and $b\geq 1001$. 
 It is easy to see that
 \begin{align}
 	\frac{\vartheta-\theta/b-\theta/a}{\ell+1-1/b-1/a}
 	&< \frac{\vartheta-\theta/a}{\ell+1-1/b-1/a}\nonumber\\
 	&< \left( 1+\frac{2}{(\ell+1)a} \right) \frac{\vartheta-\theta/a}{\ell+1}\label{thm-4-C3-ineq-1}
 \end{align}
 and
 \begin{align}
 H(\ell,a,b,\vartheta)
 	&= 1-\frac{\log a}{\log  g}-\frac{\log(\ell+1)}{\log g}-\frac{\log( \vartheta-\theta/a-\theta/b)^{-1}}{\log  g}\nonumber\\
 	&\geq 1-\frac{\log a}{\log g}-\frac{\log (\ell+1)}{\log g}-\frac{\log(\vartheta-\theta/a-\theta/1001)^{-1}}{\log g}\nonumber\\
 	&\geq 1-\frac{\log a+ \log(\ell+1)+\log(\vartheta-\theta/a-\theta/1001)^{-1}}{\log \left(1001 \ell a+1000(a-1)-1\right)}\nonumber\\
 	&:=H_{3}(\ell,a,\vartheta)\label{thm-4-C3-ineq-2}
 \end{align}
 by \eqref{thm-4-ineq-6}, $g\geq 1001\ell a+1000(a-1)-1$, and $b\geq 1001$. Hence, by \eqref{thm-4-ineq-7-2}, \eqref{thm-4-C3-ineq-1}, \eqref{thm-4-C3-ineq-2}, and the help of a computer, we conclude that
 \begin{align*}
 	\pi^{*}_{\ell,a,b} 
 	&>\left( 
 	\frac{\ell+\vartheta}{\ell+1}
 	-\frac{2}{\varphi(a)}
 	\Bigg(\sum_{\stackrel{1\leq y\leq \vartheta a}{\gcd(y,a)=1}}1
 	\Bigg)
 	H_{3}(\ell,a,\vartheta)^{-1}
 \left( 1+\frac{2}{(\ell+1)a} \right)
 \frac{\vartheta-\theta/a}{\ell+1}
 	\right)  \frac{g}{\log g}\\
 	&\geq \frac{\ell+0.02315834\ldots}{\ell+1}\frac{g}{\log g},
 \end{align*}
 and the last equality holds if $\ell=563$ and $a=564$.
 
  \textbf{Case 4.} $\max\{\ell,2\} <a\leq 15$. If $b\leq 4$ with $\gcd(a,b)=1$, then $\ell=1,2$ and $(a,b)=(3,4)$. It is easy to see that
  \begin{equation*}
  	\pi_{1,3,4}^{*}
  	=\#\{2,3,5,7,11,13,17\}
  	=\pi(g_{1,3,4})
  	>\frac{g_{1,3,4}}{\log g_{1,3,4}}
  \end{equation*}
  and
  \begin{equation*}
  	\pi_{2,3,4}^{*}=\#\{2,3,5,7,11,13,17,19,23,29\}=\pi(g_{2,3,4})>\frac{g_{2,3,4}}{\log g_{2,3,4}}.
  \end{equation*}
   If $b\geq 5$, we have $\ell ab+b\geq 1\cdot 3\cdot 5+5> 17$ and hence, by Lemma \ref{pi(x)estimate}, we have $\pi(\ell ab+b)>\frac{\ell ab+b}{\log(\ell ab+b)}$. By Lemma \ref{lem-5-1} and $\pi(\ell ab+b)>\frac{\ell ab+b}{\log(\ell ab+b)}$, we have
  \begin{align*}
  	\pi_{\ell,a,b}^{*}
  	&=\pi(g)-\#\{p\in\cP: p\leq g, p=\ell ab+ax+by,(x,y)\in\Z_{\geq 0}^{2}\}\nonumber\\
  	&\geq \pi(\ell ab+b)
  	> \frac{\ell ab+b}{\log(\ell ab+b)}
  	> \frac{\ell ab+b}{\log g}\nonumber\\
  	&= \frac{\ell ab+b}{(\ell+1) ab-a-b}\frac{g}{\log g}>\frac{\ell+1/a}{\ell+1}\frac{g}{\log g}\nonumber\\
  	&\geq \frac{\ell+1/15}{\ell+1}\frac{g}{\log g}
  	> \frac{\ell+0.066}{\ell+1}\frac{g}{\log g}.
  \end{align*}
  
  This completes the proof of Theorem \ref{thm-4}.
\end{proof}

	\section*{Acknowledgments}
Y. D. is supported by the National Natural Science Foundation of China (Grant No.
12201544) and the China Postdoctoral Scicnce Foundation (Grant No. 2022M710121). T. K. was partly supported by JSPS KAKENHI Grant Number 24K22835.

\end{document}